\documentclass[a4paper,11pt]{article}

\usepackage{mathrsfs}
\usepackage{amsmath}
\usepackage[applemac]{inputenc}
\usepackage{amsfonts}
\usepackage{amssymb}
\usepackage{amsthm}
\usepackage{stmaryrd}
\usepackage{color}
\usepackage{mathabx}
\usepackage{graphicx}
\usepackage{caption}
\usepackage{subcaption}
\usepackage{enumerate}
\usepackage{pgfplots}
\usepackage{verbatim}
\usepackage{float}
\usepackage{booktabs}
\usepackage[
    colorlinks=true,
    linkcolor=green!50!black,
    citecolor=green!50!black,
    urlcolor=blue,
    linktoc=page
]{hyperref}
\pgfplotsset{compat=1.18}
\usetikzlibrary{external}

\usepackage{cancel}                                                             
\usepackage[normalem]{ulem}

\newcommand{\E}{\mathbb{E}}

\newcommand{\Z}{\mathbb{Z}}

\renewcommand{\S}{S_{\delta}}
\renewcommand{\SS}{S_{-\delta}}

\newcommand{\R}{\mathbb{R}}

\def\1{\mbox{1\hspace{-0.25em}l}}

\def\bx{{\mathbf{x}}}

\def\p{{\mathbf{p}}}

\def\l{\lambda}
\def\<{\langle}
\def\>{\rangle}

\numberwithin{equation}{section}

\newcommand{\be}{\begin{eqnarray}}
\newcommand{\ee}{\end{eqnarray}}
\newcommand{\ce}{\begin{eqnarray*}}
\newcommand{\de}{\end{eqnarray*}}
\newtheorem{theorem}{Theorem}[section]
\newtheorem{lemma}[theorem]{Lemma}
\newtheorem{remark}[theorem]{Remark}
\newtheorem{definition}[theorem]{Definition}
\newtheorem{proposition}[theorem]{Proposition}

\newtheorem{Examples}[theorem]{Example}
\newtheorem{corollary}[theorem]{Corollary}

\def\var{{\mathrm{var}}}

\def\eps{\varepsilon}

\def\a{\alpha}

\def\p{\partial}
\def\d{\delta}
\def\g{\gamma}
\def\l{\lambda}

\def\[{{\Big[}}
\def\]{{\Big]}}
\def\<{{\langle}}
\def\>{{\rangle}}
\def\({{\big(}}
\def\){{\big)}}

\def\bx{{\mathbf{x}}}

\def\min{{\mathord{{\rm min}}}}

\def\bb2{{\boldsymbol{2}}}

\def\={&\!\!=\!\!&}

\def\1{{\mathbf{1}}}

\def\E{\mathbb E}

\def\leq{\leqslant}
\def\ge{\geqslant}
\def\le{\leqslant}

\def\iint{\int\!\!\!\int}

\def\var{{\mathrm{var}}}

\def\eps{\varepsilon}

\def\a{\alpha}

\def\p{\partial}
\def\g{\gamma}
\def\l{\lambda}

\def\[{{\Big[}}
\def\]{{\Big]}}
\def\<{{\langle}}
\def\>{{\rangle}}

\def\bx{{\mathbf{x}}}

\def\min{{\mathord{{\rm min}}}}

\def\={&\!\!=\!\!&}
\def\bt{\begin{theorem}}
\def\et{\end{theorem}}
\def\bl{\begin{lemma}}
\def\el{\end{lemma}}
\def\br{\begin{remark}}
\def\er{\end{remark}}
\def\bx{\begin{Examples}}
\def\ex{\end{Examples}}
\def\bd{\begin{definition}}
\def\ed{\end{definition}}
\def\bp{\begin{proposition}}
\def\ep{\end{proposition}}
\def\bc{\begin{corollary}}
\def\ec{\end{corollary}}

\def\leq{\leqslant}
\def\ge{\geqslant}
\def\le{\leqslant}

\def\iint{\int\!\!\!\int}

 \def\R{\mathbb R}
 \def\R{\mathbb R}

\def\<{\langle} \def\>{\rangle}

\def\0{{\mathbf{0}}}

\def\b{\beta}

\title{Numerical methods for Langevin-type SPDE: an implicit Milstein approach and multilevel Monte Carlo techniques}
\author{Sascha Portaro\thanks{Dipartimento di Matematica, (AM)$^2$, Alma Mater Studiorum - Universit\`a di Bologna, 40126 Bologna, Italy, \texttt{sascha.portaro@unibo.it}} \and Carlos V\'azquez\thanks{Department of Mathematics and CITIC, University of A Coru\~{n}a, Campus Elvi\~{n}a, 15071-A Coru\~{n}a, Spain,\texttt{carlosv@udc.es}}}
\date{\today}
\date{\today}

\begin{document}

\maketitle

\begin{abstract}
In this work, we investigate the numerical approximation of degenerate Langevin-type stochastic partial differential equations (SPDEs) in two spatial dimensions. These SPDEs arise in stochastic dynamics and mathematical finance, among other applications. In order to handle the mixed deterministic-stochastic structure of the equation and the degeneracy of the differential operator, we propose a semi-implicit Milstein finite difference scheme for the numerical solution.

Through the Fourier analysis of the mean-square stability and convergence, we derive explicit conditions on the coefficients under which the scheme is stable, jointly with explicit convergence rates in terms of the discretization parameters. We further embed the proposed scheme within a Multilevel Monte Carlo (MLMC) framework to reduce the computational cost associated with SPDE simulations, and we derive its theoretical computational complexity. 

Numerical experiments confirm theoretical convergence rates and show that the MLMC strategy achieves an accuracy comparable to standard Monte Carlo at a fraction of the computational cost, reducing the complexity from $\mathcal{O}(\varepsilon^{-5})$ to $\mathcal{O}(\varepsilon^{-3})$ for a target root-mean-square error $\varepsilon$. These results show that combining semi-implicit Milstein schemes with MLMC techniques provides an effective approach for the numerical simulation of Langevin-type SPDEs.
\end{abstract}

\section{Introduction}
Stochastic partial differential equations (SPDEs) have emerged as essential modeling frameworks in a wide range of fields, including finance, environmental sciences, biology, engineering, and spatial statistics, where phenomena exhibit both spatial-temporal dynamics and inherent stochasticity~\cite{Lord_Powell_Shardlow_2014,LiuRoeckner2015}.

The numerical analysis of parabolic SPDEs has been the subject of extensive investigation in the literature, with particular attention devoted to the construction and analysis of high-order time discretization schemes~\cite{Lord_Powell_Shardlow_2014,JentzenKloeden2011}. Notably, in \cite{Giles_Reisinger2012} the authors proposed a Milstein finite difference scheme for equations in one spatial dimension, while in \cite{Reisinger_Wang2019} the authors extended these studies to two spatial dimensions.

In this work, we consider a different class of SPDEs, specifically degenerate Langevin-type SPDEs in two spatial dimensions, which are formulated as follows:
\begin{equation}\label{SPDE}
    d_\mathbf{Y} u_t(x,v) = \frac{a_t(x,v)}{2} \partial_{vv} u_t(x,v) dt + \sigma_t(x,v) \partial_v u_t(x,v) dW_t, \qquad \mathbf{Y} := \partial_t + v \partial_x,
\end{equation}

where $W$ denotes a Wiener process defined on a complete probability space $(\Omega, \mathscr{F}, P)$ endowed with a filtration $(\mathscr{F}_t)_{t \ge 0}$ satisfying the usual conditions.

This class of equations arises naturally in mathematical finance, where \eqref{SPDE} describes the price dynamics of path-dependent financial contracts, such as Asian options \cite{Pascucci_book2011}, driven by a common stochastic factor $W$ representing the market uncertainty.

Since, even in the deterministic setting (i.e., $\sigma \equiv 0$), analytic solutions for this type of equation are generally not available, various numerical schemes have been proposed over the years. For example, explicit and implicit schemes for a bounded value problem related to the Kolmogorov equation are studied in \cite{Polidoro_Mogavero1995}, while the authors in \cite{Di_Francesco_Pascucci2004} provided both theoretical results for Kolmogorov PDEs and associated numerical methods.

More recently, in the context of degenerate SPDEs,  the authors in \cite{Kamm_Pagliarani_Pascucci2023} demonstrated how the It\^o-stochastic Magnus expansion can be employed to efficiently solve such equations in two spatial dimensions numerically.

In the present work, we will focus on the constant coefficient case, i.e., when $a$ and $\sigma$ are constant, where an analytical solution to \eqref{SPDE} is available (see \cite{Pascucci_Pesce2019, Pascucci_Pesce2022}). The availability of a closed-form solution in this case helps us to check the validity of our numerical scheme and its convergence rate, whereas the scheme itself is (hopefully) more widely applicable.

First of all, let us clarify the concept of solution of the SPDE that we will consider. Since the symbol $d_\mathbf{Y}$ indicates that the equation is solved in the It\^o (or strong) sense, a solution to \eqref{SPDE} is a continuous process $u_t = u_t(x,v)$ that is twice differentiable in $v$ and such that
\begin{equation}\label{SPDE_gamma}
    u_t \left ( \g^B_t (x,v) \right ) = u_0(x,v) + \frac{1}{2} \int_0^t{( a_s \partial_{vv} u_s ) \left ( \g^B_s (x,v) \right ) ds} + \int_0^t{(\sigma_s \partial_v u_s) \left ( \g^B_s (x,v) \right ) dW_s},
\end{equation}
where $t \mapsto \g^B_t (x,v)$ denotes the integral curve of the advection field $v \p_x$ starting from $(x,v)$, that is
\begin{equation}
    \g^B_t (x,v) = e^{tB}(x,v) = (x+tv,v), \qquad B = \left ( \begin{matrix} 0 \quad 1 \\ 0 \quad 0 \end{matrix} \right ).
\end{equation}

Let $\overline{W}, W$ be independent real Brownian motions, $a>0$ and $\sigma \in [0,\sqrt{a}]$. The Langevin model is defined in terms of the system of SDEs
\begin{equation}\label{langevin_xv}
    \begin{cases}
        dX_t = V_t dt \\
        dV_t = \sqrt{a - \sigma^2} d\overline{W}_t - \sigma dW_t.
    \end{cases}
\end{equation}
We can interpret these equations as those that govern the displacement of a particle in a phase space, where the velocity of the particle, $V_t$, is only partially observed through the process $W$, where $\overline{W}$ represents the noise in the observation.

It can be shown that the SPDE \eqref{SPDE} is the forward Kolmogorov equation of the SDE \eqref{langevin_xv} conditioned to the Brownian observation given by $\mathscr{F}_t^W = \sigma(W_s, s \le t)$, see \cite{Pascucci_Pesce2019,Pascucci_Pesce2022}, for further details.

\section{Numerical method for the Langevin SPDE}

As implementing a numerical scheme to work directly on \eqref{SPDE} can be quite challenging, then we follow the alternative approach proposed in \cite{Polidoro_Mogavero1995} and \cite{Di_Francesco_Pascucci2004}. For this purpose, recalling the definition of a solution to \eqref{SPDE}, for $\d > 0$, we have
\begin{equation}\label{SPDE_t_d}
    u_{t+\d} \left ( \g^B_\d (x,v) \right ) = u_t(x,v) + \frac{1}{2} \int_t^{t+\d}{( a_s \partial_{vv} u_s ) \left ( \g^B_{s-t} (x,v) \right ) ds} + \int_t^{t+\d}{(\sigma_s \partial_v u_s) \left ( \g^B_{s-t} (x,v) \right ) dW_s}.
\end{equation}

If we now take the drift term in implicit form and the diffusion term in explicit form, we get a semi-implicit scheme. More precisely, we approximate $u$ in the first integral of \eqref{SPDE_t_d} as $u_s(x,v) \approx u_{t+\d}(x,v)$ and thus $(a_s \p_{vv}u_s) \left ( \g^B_{s-t} (x,v) \right ) \approx (a_{t+\d} \p_{vv} u_{t+\d}) \left ( \g^B_\d (x,v) \right ) = (a_{t+\d} \p_{vv} u_{t+\d}) (x+\d v,v)$. Then we can use a zero order approximation for the diffusion term (leading to a semi-implicit Euler-Maruyama scheme) or a first order one (leading to a semi-implicit Milstein scheme).

Now, let us split the time interval $[0,T]$ into $0=t_0<t_1<\dots<t_N=T$, where $t_n = n \d$, and let us write $U^n$ to indicate the solution at time $t_n$ on a space grid.

In the semi-implicit Euler-Maruyama scheme,  we approximate $u$ in the second integral in  \eqref{SPDE_t_d} as $u_s(x,v) \approx u_{t}(x,v)$ and thus $(\sigma_s \p_{v}u_s) \left ( \g^B_{s-t} (x,v) \right ) \approx (\sigma_t \p_{v}u_t) \left ( \g^B_0 (x,v) \right ) = (\sigma_t \p_{v}u_t) (x,v)$. Therefore, the SPDE could be approximated by 
\begin{equation}
    \S U^{n+1} = U^n + \frac{\d}{2} \S K_1 U^{n+1} + \sqrt{\d} Z_{n} K_2 U^n,
\end{equation}
where $K_1$ and $K_2$ are the discretizations of the operators $a_t \p_{vv}$ and $\sigma_t \p_v$ respectively, $Z_{n} \sim \mathcal{N}_{0,1}$ are independent normal random variables, for every $n = 0,\dots,N-1$. $\S$ is the shift matrix such that $u_t(x+\d v,v) = \S u_t(x,v)$, we will be more precise in the definition of the operator $\S$ later on.

In the semi-implicit Milstein scheme,  we approximate $u$ in the second integral as $u_s(x,v) \approx u_{t}(x,v) + \sigma_t \p_v u_t (x,v) (W_s-W_t)$ and thus $(\sigma_s \p_v u_s) \left ( \g^B_{s-t} (x,v) \right ) \approx (\sigma_t \p_v u_t + \sigma_t^2 \p_{vv} u_t) (x,v) (W_s-W_t)$. Recalling that $\int_t^{t+\d}{(W_s-W_t)dW_s} = \frac{1}{2} \left ( (\Delta W^{n+1})^2 - \d \right ) = \frac{\d}{2} (Z_{n+1}^2 - 1)$, the SPDE could be approximated by 
\begin{equation}\label{Milstein_scheme}
    \S U^{n+1} = U^n + \frac{\d}{2} \S K_1 U^{n+1} + \sqrt{\d} Z_{n} K_2 U^n + \d \frac{1}{2} (Z_{n}^2 - 1) K_3 U^n,
\end{equation}
where $K_3$ is the discretization of the operator $\sigma_t^2 \p_{vv} u_t$. This scheme can be written as
\begin{equation*}
    (I - \frac{\d}{2} K_1)U^{n+1} = \SS (I + \sqrt{\d} Z_{n+1} K_2 + \d \frac{1}{2} (Z_{n+1}^2 - 1) K_3) U^n,
\end{equation*}

Let us suppose that we are working on $t \in [0,T], x \in [0,L_1], v \in [0,L_2]$, we need a grid $G$ such that $(t_n,x_j+\d v_k,v_k) =  (t_n,x_{\bar{j}},v_k) \in G$, $\forall (t_n,x_j,v_k) \in G$ with $x_j+\d v_k \le L_1$.


A possible solution to this is again to follow \cite{Polidoro_Mogavero1995}, i.e. fixing $\d$ and $h_v$, and setting $h_x := \d h_v$. In this setting, now we have that $t_n = n\d, x_j = jh_x, v_k = kh_v$ and thus $(t_n,x_j+\d v_k,v_k) = (n\d,jh_x+\d kh_v,kh_v) = (n\d,(j+k)h_x,kh_v) = (t_n,x_{j+k},v_k)$, for every $n=0,\dots,N$, $j=1,\dots,N_x$, $k=1,\dots,N_v$ such that $k+j \le N_x$.

Therefore, we can define the shift matrix as follows
\begin{equation}\label{shift_matrix}
\S := \begin{bmatrix}
S & & & \\ & S^2 & & \\ & & \ddots & \\ & & & S^{N_v-1} \\ & & & & \mathrm{0}_{N_v}
\end{bmatrix} \in \R^{N_x N_v \times N_x N_v},
\end{equation}
which can be rewritten as $\S = \sum_{j=1}^{N_v}{E^{(j)} \otimes S^j}$, where $E^{(j)}  \in \R^{N_v \times N_v}$ are zero matrices with only $E^{(j)}_{j,j}=1$ and
\begin{equation}
    S := \begin{bmatrix}
        0 & 1 & & & \\  & 0 & 1 & & \\ & & \ddots & \ddots & \\ & & & 0 & 1 \\ & & & & 0
    \end{bmatrix} \in \R^{N_x \times N_x}.
\end{equation}
Notice that the matrix $\SS$ is the transpose of $\S$.

In general $N_v \ll N_x$, since $h_x = \d h_v$ and we are working in an almost square spatial domain. Moreover, if the spatial domain is large enough, we know from the theoretical results that the SPDE solution is negligible close to the boundary, enabling us to neglect its behavior in that region.


\section{Fourier stability and convergence analysis}

\subsection{Fourier analysis of mean-square stability}
In order to perform the Fourier analysis of mean-square stability, let us first recall the definition of the Fourier transform and inverse Fourier transform pair
\begin{align*}
    \hat{u}_t(\xi,\eta) &= \int_{\R^2}{u_t(x,v) e^{-i \langle (x,v),(\xi,\eta) \rangle} dx dv}, \\
    u_t(x,v) &= \frac{1}{4 \pi^2}\int_{\R^2}{\hat{u}_t(\xi,\eta) e^{i \langle (x,v),(\xi,\eta) \rangle} d\xi d\eta}.
\end{align*}
Now, we compute the Fourier transform of the composition $u_t \circ \g_\d^B$
\begin{align*}
    \mathcal{F}(u_t \circ \g_t^B)(\xi,\eta) &= \int_{\R^2}{(u_t \circ \g_t^B)(x,v) e^{-i \langle (x,v),(\xi,\eta) \rangle} dx dv} \\
    &= \int_{\R^2}{u_t(x,v) e^{-i \langle (x - t v,v),(\xi,\eta) \rangle} dx dv} \\
    &= \int_{\R^2}{u_t(x,v) e^{-i \langle (x,v),(\xi,\eta - t \xi) \rangle} dx dv} = \hat{u}_t(\xi,\eta - t \xi).
\end{align*}
Analogously, we obtain
\begin{align*}
    \mathcal{F}((a_t \p_{vv} u_t) \circ \g_t^B)(\xi,\eta) &= \mathcal{F}(a_t \p_{vv} u_t) (\xi,\eta-t\xi), \\ 
    \mathcal{F}((\sigma_t \p_v u_t) \circ \g_t^B)(\xi,\eta) &= \mathcal{F}(\sigma_t \p_v u_t) (\xi,\eta-t\xi).
\end{align*}
In the case of constant coefficients satisfying $a > 0$ and $\sigma \in [0, \sqrt{a})$, we get
\begin{align*}
    \mathcal{F}((a \p_{vv} u_t) \circ \g_t^B)(\xi,\eta) &= -(\eta-t\xi)^2 a \hat{u}_t (\xi,\eta-t\xi), \\
    \mathcal{F}((\sigma \p_v u_t) \circ \g_t^B)(\xi,\eta) &= i (\eta-t\xi) \sigma \hat{u}_t (\xi,\eta-t\xi).
\end{align*}

By using the previous computations, the Fourier transform of the SPDE \eqref{SPDE_gamma} is given by
\begin{equation}
    d\hat{u}_t(\xi,\eta-t\xi) = \left ( -\frac{1}{2} a (\eta-t\xi)^2 dt + i \sigma (\eta-t\xi) dW_t \right ) \hat{u}_t (\xi,\eta-t\xi),
\end{equation}
so that its solution can be written as follows
\begin{equation*}
    \hat{u}_t (\xi,\eta-t\xi) = X_t(\xi,\eta) e^{-i \langle (x_0,v_0),(\xi,\eta) \rangle},
\end{equation*}
where
\begin{equation}\label{X_t}
    X_t(\xi,\eta) := \exp \left ( \frac{1}{2} (\sigma^2-a) \left ( \eta^2 t - \xi \eta t^2 + \frac{\xi^2}{3} t^3 \right ) + i \sigma (\eta - t \xi) W_t \right ).
\end{equation}
Thus, by a change of variable, we get
\begin{equation}\label{Fourier_sol}
    \hat{u}_t (\xi,\eta) = X_t(\xi,\eta + t \xi) e^{-i \langle (x_0,v_0),(\xi,\eta + t \xi) \rangle}.
\end{equation}
From \eqref{Fourier_sol}, we have $\E[|X_t|^2]\xrightarrow[t \to +\infty]{} 0$,  which implies that $\E[\| \hat{u}_t\|_{L^2}^2] \xrightarrow[t \to +\infty]{} 0$ and, by isometry, we also have $\E[\|u_t\|_{L^2}^2] \xrightarrow[t \to +\infty]{} 0$.

For the numerical solution, we use a discrete-continuous Fourier decomposition. If we denote $U^n_{j,k} = u(n\d,j h_x, k h_v)$, then
\begin{equation*}
    U^0_{j,k} = \frac{1}{4\pi^2} \int_{-\frac{\pi}{h_v}}^{\frac{\pi}{h_v}}{\int_{-\frac{\pi}{h_x}}^{\frac{\pi}{h_x}}{\hat{U}^0(w,z) e^{i \left ( (j-j_0)h_x w + (k-k_0)h_v z \right )} dw dz}},
\end{equation*}
where $j_0 = x_0/h_x, k_0 = v_0/h_v$, and
\begin{equation*}
    \hat{U}^0(\xi,\eta) := h_x h_v \sum_{j \in \Z} \sum_{k \in \Z} U^0_{j,k} e^{-i \left ( (j-j_0) h_x \xi + (k-k_0) h_v \eta \right )}.
\end{equation*}
If we set $U^0_{j,k} = h_x^{-1} h_v^{-1} \d_{(j_0,k_0)}$, we have $\hat{U}^0(w,z) = 1$ for all $(w,z) \in \R^2$. \\
Similarly, for the $n$-th step
\begin{equation*}
    U^n_{j,k} = \frac{1}{4\pi^2} \int_{-\frac{\pi}{h_v}}^{\frac{\pi}{h_v}}{\int_{-\frac{\pi}{h_x}}^{\frac{\pi}{h_x}}{\hat{U}^n(\xi,\eta) e^{i \left ( (j-j_0) h_x \xi + (k-k_0) h_v \eta \right )} d\xi d\eta}}.
\end{equation*}

By analogy to the theoretical solution \eqref{Fourier_sol}, we make the ansatz
\begin{equation}
    \hat{U}^n(\xi,\eta) = X^n(\xi,\eta + n\d \xi) \hat{U}^0(\xi,\eta) = X^n(\xi,\eta + n\d \xi),
\end{equation}
since $\hat{U}^0 \equiv 1$, where $X^n(\xi,\eta)$ can be considered the numerical approximation of $ X(n\d)$.

Analogously to the continuous case, we say that a numerical solution is asymptotically mean-square stable if for any $(\xi,\eta) \in [-\frac{\pi}{h_x},\frac{\pi}{h_x}] \times [-\frac{\pi}{h_v},\frac{\pi}{h_v}]$, we have
\begin{equation}\label{mean_square_stability}
    \lim_{n \to + \infty}{\E[|X^n(\xi,\eta)|^2]} = 0.
\end{equation}


\begin{theorem}
    Let $T > 0$, and let $\d = \frac{T}{N}$, $h_v > 0$, $h_x > 0$ be the mesh sizes, such that $h_x = \d h_v$. If $a > 0$, $\sigma \in [0, \sqrt{a})$ and $2 \sigma^4 < a^2$, then the semi-implicit Milstein finite difference scheme \eqref{Milstein_scheme} is stable in the mean-square sense of \eqref{mean_square_stability}.
\end{theorem}
\proof
First, for convenience, let us rewrite the scheme \eqref{Milstein_scheme} in the form
\begin{equation*}
    (I - \frac{\d}{2} K_1)U^{n+1} = \SS (I + \sqrt{\d} Z_{n+1} K_2 + \d \frac{1}{2} (Z_{n+1}^2 - 1) K_3) U^n,
\end{equation*}
where $K_1 = a D_{vv}, K_2 = \sigma D_v$ and $K_3 = \sigma^2 D_v^2$. 

For simplicity, next we assume that $j_0 = k_0 = 0$, and we compute
\begin{align*}
    \mathcal{F}(\SS D_v U)(\xi,\eta) &= h_x h_v \sum_{j \in \Z} \sum_{k \in \Z} (\SS \frac{1}{2 h_v} \left ( U_{j,k+1} - U_{j,k-1} \right ) ) e^{-i \left ( j h_x \xi + k h_v \eta \right )} \\
    &= h_x h_v \sum_{j \in \Z} \sum_{k \in \Z} \frac{1}{2 h_v} \left ( U_{j-k-1,k+1} - U_{j-k+1,k-1} \right ) e^{-i \left ( j h_x \xi + k h_v \eta \right )} \\
    &= h_x h_v \sum_{j \in \Z} \sum_{k \in \Z} U_{j,k} e^{-i \left ( (j+k) h_x \xi + k h_v \eta \right )} \frac{1}{2 h_v} \left ( e^{-i \eta h_v} - e^{i \eta h_v} \right ) \\
    &= h_x h_v \sum_{j \in \Z} \sum_{k \in \Z} U_{j,k} e^{-i \left ( j h_x \xi + k h_v (\eta + \d \xi) \right )} \\
    &= i q_1(\eta)\hat{U}(\xi,\eta + \d \xi), \qquad q_1(\eta) := \frac{\sin(\eta h_v)}{h_v},
\end{align*}
where we have applied the identity $e^{-i \eta h_v} - e^{i \eta h_v} = i 2 \sin(\eta h_v)$. 

Analogously, we obtain
\begin{equation*}
    \mathcal{F}(\SS D_v^2 U)(\xi,\eta) = - q_1^2(\eta) \hat{U}(\xi,\eta + \d \xi).
\end{equation*}
Moreover, we consider
\begin{align*}
    D_{vv} U_{j,k} &= \frac{1}{h_v^2} \left ( U_{j,k+1} - 2 U_{j,k} + U_{j,k-1} \right ) \\
    &= \frac{1}{4\pi^2} \int_{-\frac{\pi}{h_v}}^{\frac{\pi}{h_v}}{\int_{-\frac{\pi}{h_x}}^{\frac{\pi}{h_x}}{\hat{U}(\xi,\eta) e^{i \left ( j\xi h_x + k\eta h_v \right )} \frac{1}{h_v^2} \left ( e^{i \eta h_v} - 2 + e^{-i \eta h_v} \right ) d\xi d\eta}}; \\
\end{align*}
so that we get
\begin{align*}
    \mathcal{F} (D_{vv} U)(\xi,\eta) &= - q_2(\eta) \hat{U}(\xi,\eta), \qquad q_2(\eta) := \frac{4\sin^2(\frac{\eta h_v}{2})}{h_v^2}; \\
\end{align*}
where we have used the identity $e^{i \eta h_v} - 2 + e^{-i \eta h_v} = - 4\sin^2(\frac{\eta h_v}{2})$.

Therefore, applyng the Fourier transform to \eqref{Milstein_scheme} we obtain
\begin{equation*}
    \left ( 1 + \frac{\d}{2} a q_2(\eta) \right ) X(\xi,\eta + (n+1) \d \xi)^{n+1} = \left ( 1 + i \sqrt{\d} \sigma Z_n q_1(\eta) - \d \frac{\sigma^2}{2} \left ( Z_n^2 - 1 \right ) q_1^2(\eta) \right ) X(\xi,\eta + (n+1) \d \xi)^n,
\end{equation*}
from which we obtain the expression
\begin{equation}\label{C_n}
    X(\xi,\eta + (n+1) \d \xi)^{n+1} = C_n X(\xi,\eta + (n+1) \d \xi)^n, \qquad C_n := \frac{1 + i \sqrt{\d} \sigma Z_n q_1(\eta) - \d \frac{\sigma^2}{2} \left ( Z_n^2 - 1 \right ) q_1^2(\eta)}{1 + \frac{\d}{2} a q_2(\eta)},
\end{equation}
Recursively using the previous expression and that $X^0 \equiv 1$, we get
\begin{equation}\label{X_N}
    X(\xi,\eta + (n+1) \d \xi)^{n+1}  = \prod_{l = 0}^n C_l X(\xi,\eta + (n+1) \d \xi)^0 = \prod_{l = 0}^n C_l.
\end{equation}

Next, since the $C_n$ are independent, a sufficient condition for stability is to prove that
\begin{equation}\label{suff_cond_stability}
    \E \left [ \left | \frac{1 + i \sqrt{\d} \sigma Z_n q_1(\eta) - \d \frac{\sigma^2}{2} \left ( Z_n^2 - 1 \right ) q_1^2(\eta)}{1 + \frac{\d}{2} a q_2(\eta)} \right |^2 \right ] < 1,
\end{equation}
for any $(\xi,\eta) \in [-\frac{\pi}{h_x},\frac{\pi}{h_x}] \times [-\frac{\pi}{h_v},\frac{\pi}{h_v}] \backslash \{ (0,0) \}$.

In order to prove \eqref{suff_cond_stability}, note that
\begin{equation*}
    | 1 + \frac{\d}{2} a q_2(\eta) |^2 = 1 + \d a q_2(\eta) + \frac{\d^2}{4} a^2 q_2^2(\eta),
\end{equation*}
and
\begin{align*}
    &\left | 1 + i \sqrt{\d} \sigma Z_n q_1(\eta) - \d \frac{\sigma^2}{2} \left ( Z_n^2 - 1 \right ) q_1^2(\eta) \right |^2 \\
    &= \left ( 1 - \d \frac{\sigma^2}{2} \left ( Z_n^2 - 1 \right ) q_1^2(\eta) \right )^2 + \left ( \sqrt{\d} \sigma Z_n q_1(\eta) \right )^2 \\
    &= 1 - \d \sigma^2 \left ( Z_n^2 - 1 \right ) q_1^2(\eta) + \d^2 \frac{\sigma^4}{4} \left ( Z_n^2 - 1 \right )^2 q_1^4(\eta) + \d \sigma^2 Z_n^2 q_1^2(\eta).
\end{align*}
Next, since $\E \left [ Z_n^2 \right ] = 1$ and $\E \left [ Z_n^4 \right ] = 3$, by taking the expectation, we have
\begin{equation*}
    \E \left [ \left | 1 + i \sqrt{\d} \sigma Z_n q_1(\eta) - \d \frac{\sigma^2}{2} \left ( Z_n^2 - 1 \right ) q_1^2(\eta) \right |^2 \right ] = 1 + \d^2 \frac{\sigma^4}{2} q_1^4(\eta) + \d \sigma^2 q_1^2(\eta).
\end{equation*}
Moreover, since $\sigma^2 < a$, $2 \sigma^4 < a^2$, then the following inequalities hold
\begin{itemize}
    \item $\sigma^2 q_1^2(\eta) < a q_2(\eta)$;
    \item $\sigma^4 q_1^4(\eta) < \frac{a^2}{2} q_2^2(\eta)$.
\end{itemize}
and we end the proof.

\endproof

\subsection{Fourier analysis of convergence}
\begin{theorem}\label{convergence_thm}
    Let $T > 0, \d = \frac{T}{N}, h_v > 0, h_x = \d h_v$ be the mesh sizes. Then, if $\sigma^2 < a$ and $2\sigma^4 < a^2$, there exist $\theta \in (0,1)$, independent of $h_v$ and $\d$, such that the implicit Milstein finite difference scheme \eqref{Milstein_scheme} has the error expansion
    \begin{equation}
        U_{j,k}^N - u_T(x_j,v_k) = \d E_1(T,x_j,v_k) + h_v^2 E_2(T,x_j,v_k) + \theta^N h_v^{-2} \d^{-1} E_3(T,x_j,v_k) + o(\d,h_v^2,\theta^N h_v^{-2} \d^{-1}) R(T,x_j,v_k),
    \end{equation}
    where $x_j = j h_x, v_k = k h_v$, $E_1,E_2,E_3$ and $R$ are random variables with bounded first and second moments, all independent of $h_v$ and $\d$.
\end{theorem}

\begin{corollary}\label{L2_convergence}
    Let us assume the hypothesis of the previous Theorem. Then the discrete $L_2$ in space and $L_2$ in probability of the implicit Milstein finite difference scheme \eqref{Milstein_scheme} is
    \begin{equation}
        \left ( \sum_{i,j} h_x h_v \E \left [ \left | U^N_{i,j} - u(T,x_i,y_j) \right |^2 \right ] \right )^{\frac{1}{2}} = O(h_v^2) + O(\d) + O(h_v^{-1} \d^{-1/2} \theta^N).
    \end{equation}
    If the initial condition is in $L_2$, then
    \begin{equation*}
        \left ( \sum_{i,j} h_x h_v \E \left [ \left | U^N_{i,j} - u(T,x_i,y_j) \right |^2 \right ] \right )^{\frac{1}{2}} = O(h_v^2) + O(\d).
    \end{equation*}
\end{corollary}

In order to prove the convergence Theorem, we first need to define the following sets
\begin{equation*}
\Omega_{low} := \{ (\xi,\eta) \in \R^2 : |\eta| \le \min\{h_v^{-2p},\d^{-p}\} \}
\end{equation*}
and
\begin{equation*}
\Omega_{high} := \Omega_{low}^c \cap [-\pi h_x^{-1},\pi h_x^{-1}] \times [-\pi h_v^{-1},\pi h_v^{-1}],
\end{equation*}
for some $0<p<\frac{1}{4}$.

\begin{proposition}\label{error_small_eta}
    Fix $\d, h_v$ and $h_x := \d h_v$. Then for $(\xi,\eta) \in \Omega_{low}$, we have
    \begin{equation}
        X^N(\xi,\eta + N \d \xi) - X(T,\xi,\eta+T\xi) = X(T,\xi,\eta+T\xi) \left ( h_v^2f_1(\eta) + \d f_2(\xi,\eta) + o(\d,h_x^2,h_v^2)\varphi(T,h_x\xi,h_v\eta) \right ),
    \end{equation}
    where $f_1,f_2,\varphi$ are random variables such that after multiplication by $X(T)$, the integral over $\Omega_{low}$ has bounded first and second moments independent of N.
\end{proposition}
\proof
From \eqref{X_t}, we have
\begin{equation*}
    X_{t_{n+1}}(\xi,\eta + t_{n+1} \xi) = X_{t_n}(\xi,\eta + t_{n+1} \xi) \exp \left ( \frac{1}{2} (\sigma^2-a) \left ( \eta^2 \d + \xi \eta \d^2 + \frac{\xi^2}{3} \d^3 \right ) + i \sigma \eta \sqrt{\d} Z_n  - i \sigma \xi \d W_{t_n} \right ).
\end{equation*}
Now, if $X_n$ is the numerical approximation of $X_{\d n}(\xi,\eta + (n+1) \d \xi)$ we have
\begin{equation*}
    X_{n+1} = C_n X_n,
\end{equation*}
where
\begin{equation*}
    C_n = \exp \left ( \frac{1}{2} (\sigma^2-a) \left ( \eta^2 \d + \xi \eta \d^2 + \frac{\xi^2}{3} \d^3 \right ) + i \sigma \eta \sqrt{\d} Z_n  - i \sigma \xi \d W_{t_n} + e_n \right ),
\end{equation*}
and $e_n$ is the logarithmic error between the numerical solution and the exact solution introduced during $[n\d,(n+1)\d]$. Aggregating over $N$ time steps, at $t_N = \d N = T$,
\begin{equation*}
    X^N(\xi,\eta + N \d \xi) = X(T,\xi,\eta+T\xi) \exp{\left ( \sum_{n=0}^{N-1}{e_n} \right ) },
\end{equation*}
where $X(T)$ is the exact solution at time $T$.

Thus, we have
\begin{equation*}
    e_n = \log{C_n} - \frac{1}{2} (\sigma^2-a) \left ( \eta^2 \d + \xi \eta \d^2 + \frac{\xi^2}{3} \d^3 \right ) - i \sigma \eta \sqrt{\d} Z_n  + i \sigma \xi \d W_{t_n},
\end{equation*}
and
\begin{equation*}
    \sum_{n=0}^{N-1}{e_n} = \sum_{n=0}^{N-1}{\log{C_n}} - \frac{1}{2} (\sigma^2-a) \left ( \eta^2 + \xi \eta \d + \frac{\xi^2}{3} \d^2 \right ) T - i \sigma \eta \sqrt{\d} \sum_{n=0}^{N-1}{Z_n} + i \sigma \xi \d \sum_{n=0}^{N-1}{W_{t_n}}.
\end{equation*}

By \eqref{C_n}, we have
\begin{equation*}
    C_n = \frac{1 + i \sqrt{\d} \sigma Z_n q_1(\eta) - \d \frac{\sigma^2}{2} \left ( Z_n^2 - 1 \right ) q_1^2(\eta)}{1 + \frac{\d}{2} a q_2(\eta)},
\end{equation*}

By Taylor expansion
\begin{align*}
    \log{C_n} &= \log{\left ( 1 + i \sqrt{\d} \sigma Z_n q_1(\eta) - \d \frac{\sigma^2}{2} \left ( Z_n^2 - 1 \right ) q_1^2(\eta) \right )} - \log{\left (1 + \frac{\d}{2} a q_2(\eta) \right )} \\
    &= i \sqrt{\d} \sigma q_1(\eta) Z_n - \d \frac{\sigma^2}{2} q_1^2(\eta) Z_n^2 + \d \frac{\sigma^2}{2} q_1^2(\eta) + \d \frac{\sigma^2}{2} q_1^2(\eta) Z_n^2 - \frac{\d}{2} a q_2(\eta) \\ 
    &+ O \left ( \d \sqrt{\d} \right ) i \phi_1(Z_n) + O \left ( \d^2 \right ) \phi_2(Z_n) \\
    &= i \sqrt{\d} \sigma q_1(\eta) Z_n + \frac{\d}{2} \left ( \sigma^2 q_1^2(\eta) - a q_2(\eta) \right ) + O \left ( \d \sqrt{\d} \right ) i \phi_1(Z_n) + O \left ( \d^2 \right ) \phi_2(Z_n).
\end{align*}
Thus
\begin{align*}
    \sum_{n=0}^{N-1}{e_n} &= \sum_{n=0}^{N-1}{\log{C_n}} - \frac{1}{2} (\sigma^2-a) \left ( \eta^2 + \xi \eta \d + \frac{\xi^2}{3} \d^2 \right ) T - i \sigma \eta \sqrt{\d} \sum_{n=0}^{N-1}{Z_n} + i \sigma \xi \d \sum_{n=0}^{N-1}{W_{t_n}} \\
    &= i \left ( \sqrt{\d} \sigma (q_1(\eta) - \eta) \sum_{n=0}^{N-1}{Z_n} \right ) + \frac{1}{2} (a - \sigma^2) \left ( \eta^2 + \xi \eta \d + \frac{\xi^2}{3} \d^2 \right ) T + \frac{1}{2} \left ( \sigma^2 q_1^2(\eta) - a q_2(\eta) \right ) T \\
    &+ O \left ( \d \sqrt{\d} i \sum_{n=0}^{N-1}{\phi_1(Z_n)} \right ) + O \left ( \d^2 \sum_{n=0}^{N-1}{\phi_2(Z_n)} \right ),
\end{align*}
since $\sum_{n=0}^{N-1}{W_{t_n}} = \sqrt{\d} \sum_{n=0}^{N-1}{\sum_{l=0}^n{Z_l}}$.

Now
\begin{align*}
    &\exp{\left ( \sum_{n=0}^{N-1}{e_n} \right )} = \exp{\left ( \frac{1}{2} (a - \sigma^2) \left ( \eta^2 + \xi \eta \d + \frac{\xi^2}{3} \d^2 \right ) T + \frac{1}{2} \left ( \sigma^2 q_1^2(\eta) - a q_2(\eta) \right ) T \right )} \\
    &\cdot \exp \left ( i \left ( \sqrt{\d} \sigma (q_1(\eta) - \eta) \sum_{n=0}^{N-1}{Z_n} \right ) + O \left ( \d \sqrt{\d} i \sum_{n=0}^{N-1}{\phi_1(Z_n)} \right ) + O \left ( \d^2 \sum_{n=0}^{N-1}{\phi_2(Z_n)} \right ) \right ).
\end{align*}
We have
\begin{align*}
    q_1(\eta) & = \frac{\sin(\eta h_v)}{h_v} = \eta - \frac{1}{6} \eta^3 h_v^2 + O(\eta^5 h_v^4); \\
    q_2(\eta) & = 4 \frac{\sin^2 \left ( \frac{\eta h_v}{2} \right)}{h_v^2} = \eta^2 - \frac{1}{12} \eta^4 h_v^2 + O(\eta^6 h_v^4); \\
    q_1^2(\eta) & = \frac{\sin^2(\eta h_v)}{h_v^2} = \eta^2 - \frac{1}{3} \eta^4 h_v^2 + O(\eta^6 h_v^4).
\end{align*}
Thus
\begin{align*}
    &\exp{\left ( \frac{1}{2} (a - \sigma^2) \left ( \eta^2 + \xi \eta \d + \frac{\xi^2}{3} \d^2 \right ) T + \frac{1}{2} \left ( \sigma^2 q_1^2(\eta) - a q_2(\eta) \right ) T \right )} \\
    &= \exp{\left ( - \frac{1}{6} \left ( \sigma^2 - \frac{a}{4} \right ) \eta^4 h_v^2 T + \frac{1}{2} (a - \sigma^2) \left ( \xi \eta \d + \frac{\xi^2}{3} \d^2 \right ) T + O(\eta^6 h_v^4) \right )} \\
    &= 1 - \frac{1}{6} \left ( \sigma^2 - \frac{a}{4} \right ) \eta^4 h_v^2 T + \frac{1}{2} (a - \sigma^2) \left ( \xi \eta \d + \frac{\xi^2}{3} \d^2 \right ) T + O(\d h_v^2,\d^2, h_v^4), \\
\end{align*}
and
\begin{align*}
    &\exp \left ( i \left ( \sqrt{\d} \sigma (q_1(\eta) - \eta) \sum_{n=0}^{N-1}{Z_n} \right ) + O \left ( \d \sqrt{\d} i \sum_{n=0}^{N-1}{\phi_1(Z_n)} \right ) + O \left (  \d^2 \sum_{n=0}^{N-1}{\phi_2(Z_n)} \right ) \right ) \\
    &= 1 - i \frac{1}{6} \sqrt{\d} \sigma \eta^3 h_v^2 \sum_{n=0}^{N-1}{Z_n} + O \left ( \d \sqrt{\d} i \sum_{n=0}^{N-1}{\phi_1(Z_n)} \right ) + O \left ( \d^2 \sum_{n=0}^{N-1}{\phi_2(Z_n)} \right ).
\end{align*}
Hence
\begin{align*}
    &X^N(\xi,\eta + N \d \xi) - X(T,\xi,\eta+T\xi) = X(T,\xi,\eta+T\xi) \left ( \exp{\left ( \sum_{n=0}^{N-1} e_n \right )} - 1 \right ) \\
    &= X(T,\xi,\eta+T\xi) \Bigg ( - i \frac{1}{6} \sigma \eta^3 h_v^2 W_T - \frac{1}{6} \left ( \sigma^2 - \frac{a}{4} \right ) \eta^4 h_v^2 T  + \frac{1}{2} (a - \sigma^2) \xi \eta \d T\\
    &+ O \left ( \d \sqrt{\d} i \sum_{n=0}^{N-1}{\phi_1(Z_n)} \right ) + O \left ( \d^2 \sum_{n=0}^{N-1}{\phi_2(Z_n)} \right ) + o(\d, h_v^2) \varphi(T, h_x \xi, h_v \eta) \Bigg ).
\end{align*}

\endproof

In the case $(\xi, \eta) \in \Omega_{high}$, we have the following result.
\begin{proposition}\label{error_big_eta}
    Let $T > 0, \d = \frac{T}{N}, h_v > 0, h_x = \d h_v$ be the mesh sizes and assume that $\sigma^2 < a$ and $2\sigma^4 < a^2$. Then there exist $C > 0$ and $\theta \in (0,1)$ independent of $h_v$ and $\d$, such that, for $(\xi,\eta) \in \Omega_{high}$, we have
    \begin{equation}
        \sqrt{\E \left [ \left | \iint_{\Omega_{high}} X^N(\xi,\eta + N \d \xi) - X(T,\xi,\eta+T\xi) d\xi d\eta \right |^2 \right ]} \leq C h_v^{-2} \d^{-1} \theta^{N}.
    \end{equation}
\end{proposition}

We split the proof of the previous Proposition into steps.

First of all, let us calculate a useful bound for $\E \left [ |X_N(\xi,\eta)|^2 \right ]$.

\begin{lemma}
    For $(\xi, \eta) \in \Omega_{high}$,
    \begin{equation}\label{X_N_mean_square_estimate}
        \E \left [ |X^N(\xi,\eta + N \d \xi)|^2 \right ] \leq \left ( 1 - 4 \beta \frac{\l a \sin^2 \left ( \frac{\eta h_v}{2} \right ) + \l^2 a^2 \sin^4 \left ( \frac{\eta h_v}{2} \right )}{\left ( 1 + 2 \l a \sin^2 \left ( \frac{\eta h_v}{2} \right )  \right )^2} \right )^N,
    \end{equation}
    where $\l := \frac{\d}{h_v^2}$ and
    \begin{equation*}
        \beta := \min \left \{ 1 - \frac{\sigma^2}{a}, 1 - \frac{2 \sigma^4}{a^2} \right \} \in (0,1).
    \end{equation*}
\end{lemma}
\proof
By \eqref{X_N} we have that $X_N = X_0 \prod_{n=0}^{N-1} C_n$, where
\begin{equation*}
    C_n = \frac{1 + i \sqrt{\d} \sigma Z_n q_1(\eta) - \d \frac{\sigma^2}{2} \left ( Z_n^2 - 1 \right ) q_1^2(\eta)}{1 + \frac{\d}{2} a q_2(\eta)}.
\end{equation*}
Now, we have
\begin{align*}
    \E \left [ |C_n|^2 \right ] &= 1 - \frac{\d \left ( a q_2(\eta) - \sigma^2 q_1^2(\eta) \right ) + \d^2 \left ( \frac{a^2}{4} q_2^2(\eta) - \frac{\sigma^4}{2} q_1^4(\eta) \right )}{1 + \d a q_2(\eta) + \frac{\d^2}{4} a^2 q_2^2(\eta)} \\
    &= 1 - \frac{\l \left ( 4 a \sin^2 \left ( \frac{\eta h_v}{2} \right ) - \sigma^2 \sin^2 (\eta h_v) \right ) + \l^2 \left ( 4 a^2 \sin^4 \left ( \frac{\eta h_v}{2} \right ) - \frac{\sigma^4}{2} \sin^4 (\eta h_v) \right )}{1 + \l 4 a \sin^2 \left ( \frac{\eta h_v}{2} \right ) + \l^2 4 a^2 \sin^4 \left ( \frac{\eta h_v}{2} \right )} \\
    &= 1 - \frac{4 \l \sin^2 \left ( \frac{\eta h_v}{2} \right ) \left ( a - \sigma^2 \cos^2 \left ( \frac{\eta h_v}{2} \right ) \right ) + 4 \l^2 \sin^4 \left ( \frac{\eta h_v}{2} \right ) \left ( a^2 - 2 \sigma^4 \cos^4 \left ( \frac{\eta h_v}{2} \right ) \right )}{1 + \l 4 a \sin^2 \left ( \frac{\eta h_v}{2} \right ) + \l^2 4 a^2 \sin^4 \left ( \frac{\eta h_v}{2} \right )} \\
    &\leq 1 - \frac{4 \l a \sin^2 \left ( \frac{\eta h_v}{2} \right ) \left ( 1 - \frac{\sigma^2}{a} \right ) + 4 \l^2 a^2 \sin^4 \left ( \frac{\eta h_v}{2} \right ) \left ( 1 - \frac{2 \sigma^4}{a^2} \right )}{1 + \l 4 a \sin^2 \left ( \frac{\eta h_v}{2} \right ) + \l^2 4 a^2 \sin^4 \left ( \frac{\eta h_v}{2} \right )}.
\end{align*}
Let us define $d := a \sin^2 \left ( \frac{\eta h_v}{2} \right )$ and $\b := \min \left \{ 1 - \frac{\sigma^2}{a}, 1 - \frac{2 \sigma^4}{a^2} \right \} \in (0,1)$, since by hypothesis $\sigma^2 < a$ and $2 \sigma^4 < a^2$. We get
\begin{align*}
    \E \left [ |C_n|^2 \right ] \leq 1 - \frac{4 \b (\l d + \l^2 d^2)}{(1 + 2 \l d)^2},
\end{align*}
and thus \eqref{X_N_mean_square_estimate} holds.
\endproof

Now, we consider the case $0 < \l < 1$, and thus
\begin{equation*}
\Omega_{high} = \{ (\xi,\eta) \in \R^2 : |\eta| > h_v^{-2p} \} \cap [-\pi h_x^{-1},\pi h_x^{-1}] \times [-\pi h_v^{-1},\pi h_v^{-1}],
\end{equation*}
in particular $|\Omega_{high}| = 4 \pi h_x^{-1}(\pi h_v^{-1} - h_v^{-2p})$.

\begin{lemma}\label{error_big_eta_pt1}
    For $0 < \l < 1$, (i.e. $\d \leq h_v^2$),
    \begin{equation}
        \E \left [ \left | \iint_{\Omega_{high}} X^N(\xi,\eta + N \d \xi) - X(T,\xi,\eta+T\xi) d\xi d\eta \right |^2 \right ] = o(h_v^r), \quad \forall r > 0.
    \end{equation}
\end{lemma}
\proof
We have
\begin{equation*}
    X(T,\xi,\eta) = \exp \left ( \frac{1}{2} (\sigma^2-a) \left ( \eta^2 T - \xi \eta T^2 + \frac{\xi^2}{3} T^3 \right ) + i \sigma (\eta - T \xi) W_T \right ).
\end{equation*}
Then
\begin{align*}
    &\E \left [ \left | \iint_{\Omega_{high}} X^N(\xi,\eta + N \d \xi) - X(T,\xi,\eta+T\xi) d\xi d\eta \right |^2 \right ] \\
    &\leq 4 \pi^2 h_x^{-1} h_v^{-1} \E \left [ \iint_{\Omega_{high}} \left | X^N(\xi,\eta + N \d \xi) - X(T,\xi,\eta+T\xi) \right |^2 d\xi d\eta \right ] \\
    &\leq 8 \pi^2 h_x^{-1} h_v^{-1} \iint_{\Omega_{high}} \E \left [ \left | X^N(\xi,\eta + N \d \xi) \right |^2 + \left | X(T,\xi,\eta+T\xi) \right |^2 \right ] d\xi d\eta \\
    &= 8 \pi^2 h_x^{-1} h_v^{-1} \iint_{\Omega_{high}}{ \E \left [ \left | X^N(\xi,\eta + N \d \xi) \right |^2 \right ] d\xi d\eta} + f_0(\d),
\end{align*}
where
\begin{align*}
    f_0(\d) &= 8 \pi^2 h_x^{-1} h_v^{-1} \iint_{\Omega_{high}}{ \exp \left ( (\sigma^2-a) \left ( (\eta+T\xi)^2 T - \xi (\eta+T\xi) T^2 + \frac{\xi^2}{3} T^3 \right ) \right ) d\xi d\eta } \\
    &= 8 \pi^2 h_x^{-1} h_v^{-1} \iint_{\Omega_{high}} {\exp \left ( (\sigma^2-a) \left ( \eta^2 T + \xi \eta T^2 + \frac{\xi^2}{3} T^3 \right ) \right ) d\eta d\xi} \\
    &\leq C h_x^{-2} h_v^{-2} \exp \left ( (\sigma^2-a) C h_v^{-4p} \right ) d\eta d\xi = o(h_v^r), \quad \forall r > 0,
\end{align*}
since $\min_{\Omega_{high}}{\left ( \eta^2 T + \xi \eta T^2 + \frac{\xi^2}{3} T^3 \right )} \ge C h_v^{-4p}$, with $C > 0$.

From \eqref{X_N_mean_square_estimate}, we have
\begin{equation*}
    \E \left [ |X^N(\xi,\eta + N \d \xi)|^2 \right ] \leq \left ( 1 - \frac{4 \b (d + \l d^2)}{(1 + 2 \l d)^2} \frac{T}{N h_v^2} \right )^N < \exp{\left ( -\frac{4 \b (d + \l d^2) T}{(1 + 2 \l d)^2} h_v^{-2} \right )}
\end{equation*}
where
\begin{equation*}
    d := a \sin^2 \left ( \frac{\eta h_v}{2} \right ) \ge a \sin^2 \left ( \frac{h_v^{1-2p}}{2} \right ) = a \left ( \frac{h_v^{2-4p}}{4} - \frac{h_v^{4-8p}}{48} + O (h_v^{1-2p}) \right ),
\end{equation*}
since $|\eta| \in [h_v^{-2p},\pi h_v^{-1}]$.

Therefore,
\begin{align*}
    \E \left [ |X^N(\xi,\eta + N \d \xi)|^2 \right ]  &< \exp{\left ( -\frac{4 \b (d + \l d^2) T}{(1 + 2 \l d)^2} h_v^{-2} \right )} \\
    &< \exp{\left ( -4\b d T h_v^{-2} \right )} < \exp{\left ( -a\b T h_v^{-4p} \right )} = o(h_v^r), \quad \forall r > 0.
\end{align*}
This concludes the proof.
\endproof

For $\l \ge 1$, we split $\Omega_{high}$ into two region $$\Omega_1 := \{ ( |\xi|, |\eta|) \in [0, \pi h_x^{-1}] \times [\d^{-p}, \d^{-\frac{1}{2}}] \},$$ and $$\Omega_2 := \{ ( |\xi|, |\eta|) \in [0, \pi h_x^{-1}] \times [\d^{-\frac{1}{2}}, \pi h_v^{-1}] \}.$$

\begin{lemma}\label{error_big_eta_pt2}
    For $\l \ge 1$ (i.e. $\d \ge h_v^2$), there exists $\theta \in (0,1)$ independent of $h_v$ and $\d$ such that
    \begin{equation}
        \E \left [ \left | \iint_{\Omega_{high}} X^N(\xi,\eta + N \d \xi) - X(T,\xi,\eta+T\xi) d\xi d\eta \right |^2 \right ] \leq C h_v^{-4} \d^{-2} \theta^{2N}.
    \end{equation}
\end{lemma}
\proof
\begin{align*}
    \E &\left [ \left | \iint_{\Omega_{high}} X^N(\xi,\eta + N \d \xi) - X(T,\xi,\eta+T\xi) d\xi d\eta \right |^2 \right ] \\
    &\leq 2 \E \left [ \left | \iint_{\Omega_1} X^N(\xi,\eta + N \d \xi) - X(T,\xi,\eta+T\xi) d\xi d\eta \right |^2 \right ] \\
    &+ 2 \E \left [ \left | \iint_{\Omega_2} X^N(\xi,\eta + N \d \xi) - X(T,\xi,\eta+T\xi) d\xi d\eta \right |^2 \right ] \\
    &\leq 8 \pi h_x^{-1} \d^{-\frac{1}{2}} \E \left [ \iint_{\Omega_1} \left |X^N(\xi,\eta + N \d \xi) - X(T,\xi,\eta+T\xi) \right |^2 d\xi d\eta \right ] \\
    &+ 8 \pi^2 h_x^{-1} h_v^{-1} \E \left [ \iint_{\Omega_2} \left | X^N(\xi,\eta + N \d \xi) - X(T,\xi,\eta+T\xi) \right |^2 d\xi d\eta \right ] \\
    &\leq 16 \pi h_x^{-1} \d^{-\frac{1}{2}} \E \left [ \iint_{\Omega_1} \left |X^N(\xi,\eta + N \d \xi) \right |^2 + \left | X(T,\xi,\eta+T\xi) \right |^2 d\xi d\eta \right ] \\
    &+ 16 \pi^2 h_x^{-1} h_v^{-1} \E \left [ \iint_{\Omega_2} \left | X^N(\xi,\eta + N \d \xi) \right |^2 + \left | X(T,\xi,\eta+T\xi) \right |^2 d\xi d\eta \right ] \\
    &= 16 \pi h_x^{-1} \d^{-\frac{1}{2}} \E \left [ \iint_{\Omega_1} \left |X^N(\xi,\eta + N \d \xi) \right |^2 d\xi d\eta \right ] \\
    &+ 16 \pi^2 h_x^{-1} h_v^{-1} \E \left [ \iint_{\Omega_2} \left | X^N(\xi,\eta + N \d \xi) \right |^2 d\xi d\eta \right ] + f_1(\d),
\end{align*}
where
\begin{align*}
    f_1(\d) &:= 16 \pi h_x^{-1} \d^{-\frac{1}{2}} \iint_{\Omega_1} \exp \left ( (\sigma^2-a) \left ( \eta^2 T + \xi \eta T^2 + \frac{\xi^2}{3} T^3 \right ) \right ) d\xi d\eta \\
    &+ 16 \pi^2 h_x^{-1} h_v^{-1} \iint_{\Omega_2} \exp \left ( (\sigma^2-a) \left ( \eta^2 T + \xi \eta T^2 + \frac{\xi^2}{3} T^3 \right ) \right ) d\xi d\eta \\
    &\leq C \pi h_x^{-2} \d^{-1} \exp \left ( -a \b C \d^{-2p} \right ) + C h_x^{-2} h_v^{-2} \exp \left ( -a \b C \d^{-1} \right ),
\end{align*}
since $\min_{\Omega_1}{\left ( \eta^2 T + \xi \eta T^2 + \frac{\xi^2}{3} T^3 \right )} \ge C \d^{-1}$, and $\min_{\Omega_2}{\left ( \eta^2 T + \xi \eta T^2 + \frac{\xi^2}{3} T^3 \right )} \ge C \d^{-\frac{1}{2}}$ , $C > 0$.

From \eqref{X_N_mean_square_estimate}, we have
\begin{equation*}
    \E \left [ |X^N(\xi,\eta + N \d \xi)|^2 \right ] \leq \left ( 1 - \frac{4 \b (d + \l d^2)}{(1 + 2 \l d)^2} \frac{T}{N h_v^2} \right )^N < \exp{\left ( -\frac{4 \b (d + \l d^2) T}{(1 + 2 \l d)^2} h_v^{-2} \right )}
\end{equation*}
where
\begin{equation*}
    d := a \sin^2 \left ( \frac{\eta h_v}{2} \right ) \ge a \sin^2 \left ( \frac{\d^{-p} h_v}{2} \right ) = a \left ( \frac{\d^{-2p} h_v^2}{4} - \frac{\d^{-4p} h_v^4}{48} + O (\d^{-5p} h_v^5) \right ),
\end{equation*}
since $|\eta| \in [\d^{-p},\d^{-\frac{1}{2}}]$.

Therefore,
\begin{align*}
    \E \left [ |X^N(\xi,\eta + N \d \xi)|^2 \right ]  &< \exp{\left ( -\frac{4 \b (d + \l d^2) T}{(1 + 2 \l d)^2} h_v^{-2} \right )} \\
    &< \exp{\left ( -4\b d T h_v^{-2} \right )} < \exp{\left ( -a\b T \d^{-2p} \right )},
\end{align*}
and
\begin{align*}
    16 \pi h_x^{-1} \d^{-\frac{1}{2}} \E \left [ \iint_{\Omega_1} \left |X^N(\xi,\eta + N \d \xi) \right |^2 d\xi d\eta \right ] \leq 64 \pi^2 h_x^{-2} \d^{-1} \exp{\left ( -a\b T \d^{-2p} \right )}.
\end{align*}
For $(\xi, \eta) \in \Omega_2$, $d \in [a \sin{\left ( \frac{1}{2 \sqrt{\l}} \right )}, a]$,
\begin{align*}
    \max_{d}{\left ( 1 - \frac{4 \b (\l d + \l^2 d^2)}{(1 + 2 \l d)^2} \right )} &= 1 - \b \min_{d}{\left ( 1 - \frac{1}{(1 + 2 \l d)^2} \right )} \\
    &= 1 - \b \left ( 1 - \max_{d}{\frac{1}{(1 + 2 \l d)^2}} \right ).
\end{align*}
Now, we have that
\begin{equation*}
    \max_{d}{\frac{1}{(1 + 2 \l d)^2}} = \frac{1}{ \left (1 + 2 \l a \sin{\left ( \frac{1}{2 \sqrt{\l}} \right )} \right )^2 } =: \a \in (0,1),
\end{equation*}
and thus
\begin{equation*}
    1 - \frac{4 \b (\l d + \l^2 d^2)}{(1 + 2 \l d)^2} \leq 1 - \a \b =: \theta_0 \in (0,1).
\end{equation*}
Then
\begin{equation*}
    \E \left [ |X^N(\xi,\eta + N \d \xi)|^2 \right ]  < \left ( 1 - \frac{4 \b (\l d + \l^2 d^2)}{(1 + 2 \l d)^2} \right )^N \leq \theta_0^N.
\end{equation*}
So
\begin{equation*}
    16 \pi^2 h_x^{-1} h_v^{-1} \E \left [ \iint_{\Omega_2} \left | X^N(\xi,\eta + N \d \xi) \right |^2 d\xi d\eta \right ] \leq 64 \pi^4 h_x^{-2} h_v^{-2} \theta_0^N.
\end{equation*}
Hence
\begin{align*}
    \E &\left [ \left | \iint_{\Omega_{high}} X^N(\xi,\eta + N \d \xi) - X(T,\xi,\eta+T\xi) d\xi d\eta \right |^2 \right ] \\
    &\leq C \pi h_x^{-2} \d^{-1} \exp \left ( -a \b C \d^{-2p} \right ) + C h_x^{-2} h_v^{-2} \exp \left ( -a \b C \d^{-1} \right ) \\
    &+ 64 \pi^2 h_x^{-2} \d^{-1} \exp{\left ( -a\b T \d^{-2p} \right )} + 64 \pi^4 h_x^{-2} h_v^{-2} \theta_0^N \\
    &= C \pi h_v^{-2} \d^{-3} \exp \left ( -a \b C \d^{-2p} \right ) + C h_v^{-4} \d^{-2} \exp \left ( -a \b C \d^{-1} \right ) \\
    &+ 64 \pi^2 h_v^{-2} \d^{-3} \exp{\left ( -a\b T \d^{-2p} \right )} + 64 \pi^4 h_v^{-4} \d^{-2} \theta_0^N.
\end{align*}
Since the first three terms in the previous inequality are terms of order higher than $h_v^{-4} \d^{-2} \theta_0^N$, we have, for $\l \ge 1$,
\begin{equation*}
    \E \left [ \left | \iint_{\Omega_{high}} X^N(\xi,\eta + N \d \xi) - X(T,\xi,\eta+T\xi) d\xi d\eta \right |^2 \right ] \leq C h_v^{-4} \d^{-2} \theta_0^N.
\end{equation*}
Letting $\theta := \sqrt{\theta_0}$, the thesis follows.
\endproof

\proof[Proof of Proposition \ref{error_big_eta}]
The result is a consequence of Lemma \ref{error_big_eta_pt1} and Lemma \ref{error_big_eta_pt2}.
\endproof

Now we are ready to prove Theorem \ref{convergence_thm}.

\proof[Proof of Theorem \ref{convergence_thm}]
The inverse Fourier transform gives
\begin{align*}
    U_{j,k}^N &= \frac{1}{4\pi^2} \int_{-\frac{\pi}{h_v}}^{\frac{\pi}{h_v}}{\int_{-\frac{\pi}{h_x}}^{\frac{\pi}{h_x}}{\hat{U}^N(\xi,\eta) e^{i \left ( (j-j_0) h_x \xi + (k-k_0) h_v \eta \right )} d\xi d\eta}} \\
    &= \frac{1}{4\pi^2} \int_{-\frac{\pi}{h_v}}^{\frac{\pi}{h_v}}{\int_{-\frac{\pi}{h_x}}^{\frac{\pi}{h_x}}{X^N(\xi,\eta + N \d \xi) e^{i \left ( (j-j_0) h_x \xi + (k-k_0) h_v \eta \right )} d\xi d\eta}}
\end{align*}
and
\begin{align*}
    u_T(x_j,v_k) &= \frac{1}{4 \pi^2}\int_{\R^2}{\hat{u}_T(\xi,\eta) e^{i \langle (x_j,v_k),(\xi,\eta) \rangle} d\xi d\eta} \\
    &= \frac{1}{4 \pi^2}\int_{\R^2}{X(T,\xi,\eta+T\xi) e^{i \left ( (j-j_0) h_x \xi + (k-k_0) h_v \eta \right )} d\xi d\eta}.
\end{align*}
Hence, we have
\begin{equation*}
    U_{j,k}^N - u_T(x_j,v_k) = \frac{1}{4 \pi^2}\iint_{\Omega_{low} \cup \Omega_{high}}{\left ( X^N(\xi,\eta + N \d \xi) - X(T,\xi,\eta+T\xi) \right ) e^{i \langle (x_j-x_0,v_k-v_0),(\xi,\eta) \rangle} d\xi d\eta},
\end{equation*}
and the thesis follows by Proposition \ref{error_small_eta} and Proposition \ref{error_big_eta}.
\endproof

\proof[Proof of Corollary \ref{L2_convergence}]
By Parseval's Theorem, we have
\begin{equation*}
    \sum_{i,j} h_x h_v \left | U^N_{i,j} - u(T,x_i,y_j) \right |^2 = \iint_{\Omega_{low} \cup \Omega_{high}}{|\hat{u}(0,\xi,\eta + T \xi)|^2 |X^N(\xi,\eta + N \d \xi) - X(T,\xi,\eta+T\xi)|^2 d\xi d\eta}.
\end{equation*}
By Proposition \ref{error_small_eta}, we have $|X^N(\xi,\eta + N \d \xi) - X(T,\xi,\eta+T\xi)| = |X(T,\xi,\eta+T\xi)| \left ( O(h_v^2) + O(\d) \right )$, for $(\xi,\eta) \in \Omega_{low}$. By Proposition \ref{error_big_eta}, we have $|X^N(\xi,\eta + N \d \xi) - X(T,\xi,\eta+T\xi)|^2 \leq C \theta^{2N}$, for $(\xi,\eta) \in \Omega_{high}$ with $C>0$.

Now, if the initial datum is a Dirac delta function, then $|\hat{u}(0,\xi,\eta + T \xi)| = 1$, and we have 
\begin{align*}
    &\iint_{\Omega_{low} \cup \Omega_{high}}{|\hat{u}(0,\xi,\eta + T \xi)|^2 |X^N(\xi,\eta + N \d \xi) - X(T,\xi,\eta+T\xi)|^2 d\xi d\eta} \\
    &= \iint_{\Omega_{low}}{|X^N(\xi,\eta + N \d \xi) - X(T,\xi,\eta+T\xi)|^2 d\xi d\eta} + \iint_{\Omega_{high}}{|X^N(\xi,\eta + N \d \xi) - X(T,\xi,\eta+T\xi)|^2 d\xi d\eta} \\
    &= O(h_v^4) + O(\d^2) + O(h_v^{-2} \d^{-1} \theta^{2N}).
\end{align*}
If the initial datum is in $L_2$, then $\int \int{|\hat{u}(0,\xi,\eta + T \xi)|^2 d\xi d\eta} < + \infty$ and
\begin{equation*}
    \iint_{\Omega_{high}}{|\hat{u}(0,\xi,\eta + T \xi)|^2 |X^N(\xi,\eta + N \d \xi) - X(T,\xi,\eta+T\xi)|^2 d\xi d\eta} = O(\d^r), \qquad \text{for any } r > 0,
\end{equation*}
so
\begin{equation*}
    \iint_{\Omega_{low} \cup \Omega_{high}}{|\hat{u}(0,\xi,\eta + T \xi)|^2 |X^N(\xi,\eta + N \d \xi) - X(T,\xi,\eta+T\xi)|^2 d\xi d\eta} = O(h_v^4) + O(\d^2).
\end{equation*}
\endproof

\section{Multi-level Monte Carlo}
In this section, we consider the problem of estimating the expectation of a Lipschitz function $P$ of the solution of the SPDE \eqref{SPDE}.

Suppose that we want to approximate $P$ using a standard Monte Carlo (MC) approach (i.e. averaging over $N$ independent SPDE simulations). To achieve a $RMSE$ of order $\eps$ we would need $N = O(\eps^{-2})$, $h_v = O(\eps^{\frac{1}{2}})$, $\d = O(\eps)$ and consequently $h_x = O(\eps^{\frac{3}{2}})$, so the computational cost would be $O(N h_x^{-1} h_v^{-1} \d^{-1}) = O(\eps^{-5})$.

Since the complexity of a standard MC is quite high, we aim to improve it by using a Multi-level Monte Carlo (MLMC) strategy. Following a standard approach (see for example~\cite{Giles_Reisinger2012, Giles2008a}), we obtain this result.
\begin{theorem}\label{MLMC_theorem}
    Let $P$ denote a functional of the solution of the SPDE \eqref{SPDE} for a given Brownian path $W_t$, and let $\hat{P}_l$ denote the corresponding level $l$ numerical approximation.

    If there exist independent estimators $\hat{Y}_l$ based on $N_l$ Monte Carlo samples, and positive constants $\a,\b,\g,c_1,c_2,c_3$, such that
    \begin{enumerate}[(i)]
        \item $|\E [\hat{P}_l - P]| \leq c_1 h^{\a l}$; \label{i}
        \item $\E [\hat{Y}_l] = \begin{cases} \E [\hat{P}_0], & l = 0, \\ \E [\hat{P}_l - \hat{P}_{l-1}] & l > 0; \end{cases}$ \label{ii}
        \item $\var(\hat{Y}_l) \leq c_2 N_l^{-1} h^{\b l}$; \label{iii}
        \item $C_l \leq c_3 N_l h^{-\g l}$, where $C_l$ is the computational complexity of $\hat{Y}_l$; \label{iv}
    \end{enumerate}
    then there exists a positive constant $c_4$ such that there are values of $L$ and $N_l$ for which the multilevel estimator
    \begin{equation*}
        \hat{Y} = \sum_{l=0}^L{\hat{Y}_l}
    \end{equation*}
    has a mean-square error with bound
    \begin{equation*}
        MSE := \E \left [ \left ( \hat{Y} - \E [P] \right )^2 \right ] < \eps^2
    \end{equation*}
    with a computational complexity $C$ with bound
    \begin{equation*}
        C \leq c_4 \left ( \eps^{-2 - \frac{\g-\b}{\a}} + \eps^{- \frac{\g}{\a}} \right ), \qquad \text{for } 0 < \b < \g.
    \end{equation*}
\end{theorem}

\proof
Let $N > 0, h_l := \frac{T}{N^l}$, and $L = \left \lceil \frac{\log \left ( \sqrt{2} c_1 T^\a \eps^{-1} \right )}{\a \log(N)} \right \rceil$, (where $\lceil x \rceil$ denotes the integer part of a real number $x$). Then, we have
\begin{equation*}
    N^{-\a} \frac{1}{\sqrt{2}} \eps < c_1 h_L^\a \leq \frac{1}{\sqrt{2}} \eps,
\end{equation*}
and, by \eqref{i} and \eqref{ii}
\begin{equation}\label{MLMC_mean_estimate}
    \left ( \E [\hat{Y}] - \E [P] \right )^2 \leq \frac{1}{2} \eps^2.
\end{equation}
Moreover, we have the following useful inequality
\begin{equation*}
    \sum_{l=0}^L{h_l^{-\g}} = h_L^{-\g} \sum_{l=0}^L{N^{-\g l}} < h_L^{-\g} \frac{N^\g}{N^\g -1},
\end{equation*}
and
\begin{equation}\label{hL_estimate}
    h_L^{-\g} < N^\g \left ( \frac{\eps}{\sqrt{2} c_1} \right )^{-\frac{\g}{\a}}.
\end{equation}
Hence
\begin{equation}\label{sum_hl_estimate}
    \sum_{l=0}^L{h_l^{-\g}} < \frac{N^{2\g}}{N^\g -1} \left ( \sqrt{2} c_1 \right )^{\frac{\g}{\a}} \eps^{-\frac{\g}{\a}}.
\end{equation}
Now, let us set $N_l := \left \lceil 2 \eps^{-2} c_2 h_L^{-\frac{\g-\b}{2}} (1-N^{-\frac{\g-\b}{2}})^{-1} h_l^{\frac{\g+\b}{2}} \right \rceil$. Then, by \eqref{iii}
\begin{equation*}
    \sum_{l=0}^L{\var(\hat{Y}_l)} < \frac{1}{2} \eps^2 h_L^{\frac{\g-\b}{2}} (1-N^{-\frac{\g-\b}{2}}) \sum_{l=0}^L h_l^{-\frac{\g-\b}{2}} < \frac{1}{2} \eps^2,
\end{equation*}
since
\begin{equation*}
    \sum_{l=0}^L h_l^{-\frac{\g-\b}{2}} < h_L^{-\frac{\g-\b}{2}} (1-N^{-\frac{\g-\b}{2}})^{-1}.
\end{equation*}
Therefore, by the last inequality and the estimate \eqref{MLMC_mean_estimate}, we have an $\eps^2$ upper bound for the $MSE$.

Finally, we have
\begin{equation*}
    N_l < 2 \eps^{-2} c_2 h_L^{-\frac{\g-\b}{2}} (1-N^{-\frac{\g-\b}{2}})^{-1} h_l^{\frac{\g+\b}{2}} + 1,
\end{equation*}
so, by \eqref{iv} and \eqref{sum_hl_estimate}, the computational complexity is bounded by
\begin{align*}
    C &\leq c_3 \sum_{l=0}^L N_l h_l^{-\g} \\
    &< c_3 \left ( 2 \eps^{-2} c_2 h_L^{-\frac{\g-\b}{2}} (1-N^{-\frac{\g-\b}{2}})^{-1} \sum_{l=0}^L h_l^{-\frac{\g-\b}{2}} + \sum_{l=0}^L h_l^{-\g} \right ) \\
    &< c_3 \left ( 2 \eps^{-2} c_2 (1-N^{-\frac{\g-\b}{2}})^{-2} h_L^{-(\g-\b)} + \frac{N^{2\g}}{N^\g -1} \left ( \sqrt{2} c_1 \right )^{\frac{\g}{\a}} \eps^{-\frac{\g}{\a}} \right ).
\end{align*}
Now, by \eqref{hL_estimate} we have $h_L^{-(\g-\b)} < N^{\g-\b} (\sqrt{2} c_1)^{\frac{\g-\b}{\a}} \eps^{-\frac{\g-\b}{\a}}$. Thus
\begin{align*}
    C &< c_3 \left ( 2 c_2 (1-N^{-\frac{\g-\b}{2}})^{-2} N^{\g-\b} (\sqrt{2} c_1)^{\frac{\g-\b}{\a}} \eps^{-2 - \frac{\g-\b}{\a}} + \frac{N^{2\g}}{N^\g -1} \left ( \sqrt{2} c_1 \right )^{\frac{\g}{\a}} \eps^{-\frac{\g}{\a}} \right ) \\
    &\leq c_4 \left ( \eps^{-2 - \frac{\g-\b}{\a}} + \eps^{-\frac{\g}{\a}} \right ),
\end{align*}
where $c_4 := c_3 \cdot \max{\{ 2 c_2 (1-N^{-\frac{\g-\b}{2}})^{-2} N^{\g-\b} (\sqrt{2} c_1)^{\frac{\g-\b}{\a}}, \frac{N^{2\g}}{N^\g -1} \left ( \sqrt{2} c_1 \right )^{\frac{\g}{\a}} \}}$.
\endproof

Now, let us set $h_v \approx h^l$, $\d \approx h^{2l}$ and $h_x \approx h^{3l}$ for the level $l$ in the MLMC, for some $h<1$. Then, the computational cost at level $l$ is $C_l \leq c_3 N_l h^{-6l}$, i.e. $\g = 6$.

Moreover, if $P$ is a Lipschitz function of the SPDE solution, then Corollary \ref{L2_convergence} implies that the weak and strong errors for the level $l$ satisfy
\begin{equation*}
    |\E [\hat{P}_l - P]| \leq c_1 h^{2 l},
\end{equation*}
and
\begin{align*}
    \var{(\hat{P}_l - \hat{P}_{l-1})} &\leq  \var{(\hat{P}_l - P)} + \var{(\hat{P}_{l-1} - P)} \\
    &\leq \E [(\hat{P}_l - P)^2] + \E [(\hat{P}_{l-1} - P)^2] \leq c_2 h^{4l},
\end{align*}
so $\a = 2$ and $\b = 4 < \g$. Notice that we are not in the standard case $\a > \frac{1}{2} \g$.

Therefore, following a MLMC strategy, Theorem \ref{MLMC_theorem} assures that the complexity for achieving a $RMSE$ of order $\eps$ is $C \leq c_4 \left ( \eps^{-2 - \frac{\g-\b}{\a}} + \eps^{- \frac{\g}{\a}} \right ) = 2 c_4 \eps^{-3}$.

\section{Numerical tests}
We numerically verify the order of convergence by approximating the $L_2$ error
\begin{equation*}
\left ( \sum_{i,j} h_x h_v \E \left [ \left | U^N_{i,j} - u(T,x_i,y_j) \right |^2 \right ] \right )^{\frac{1}{2}}
\end{equation*}
by
\begin{equation}
    \label{L_2_err_approx}
    \left ( \sum_{i,j} h_x h_v E_L \left [ \left | U^N_{i,j} - u(T,x_i,y_j) \right |^2 \right ] \right )^{\frac{1}{2}} := \left ( \sum_{i,j} \frac{h_x h_v}{L} \sum_{l=1}^L \left | U^N_{i,j}(W^{(l)}) - u(T,x_i,y_j;W^{(l)}) \right |^2 \right )^{\frac{1}{2}},
\end{equation}
where $W^{(l)}$ are independent Brownian motions and $E_L$ is the empirical mean with $L$ samples.


Figure \ref{fig1} illustrates the empirical strong error between numerical solutions computed on consecutively refined grids. For each parameter (either $\delta$ or $h_v$), we consider a sequence of discretizations and calculate the differences between successive approximations, resulting in a set of points that quantify the convergence behavior of the numerical method as the grid is refined. The results show that the observed convergence rates are consistent with the theoretical analysis, exhibiting first-order behavior with respect to $\delta$ and second-order behavior with respect to $h_v$. This provides a clear visual confirmation that the numerical method achieves the expected accuracy as predicted by the theoretical estimates.

Figure \ref{fig2} displays the numerical approximation of the solution at the final time $T = 1$ together with the corresponding pointwise error with respect to the exact solution. In the left panel we report the approximate solution obtained with discretization parameters $h_v = 10^{-1}$ and $\d = 10^{-2}$, while the right panel shows the difference between the approximate and exact solution on the same grid. The initial condition is centered at $x_0 = 0.1$ and $v_0 = 5$. The comparison highlights both the structure of the numerical approximation and the regions where the error is most pronounced, providing a clear indication of the accuracy of the scheme under the chosen discretization.

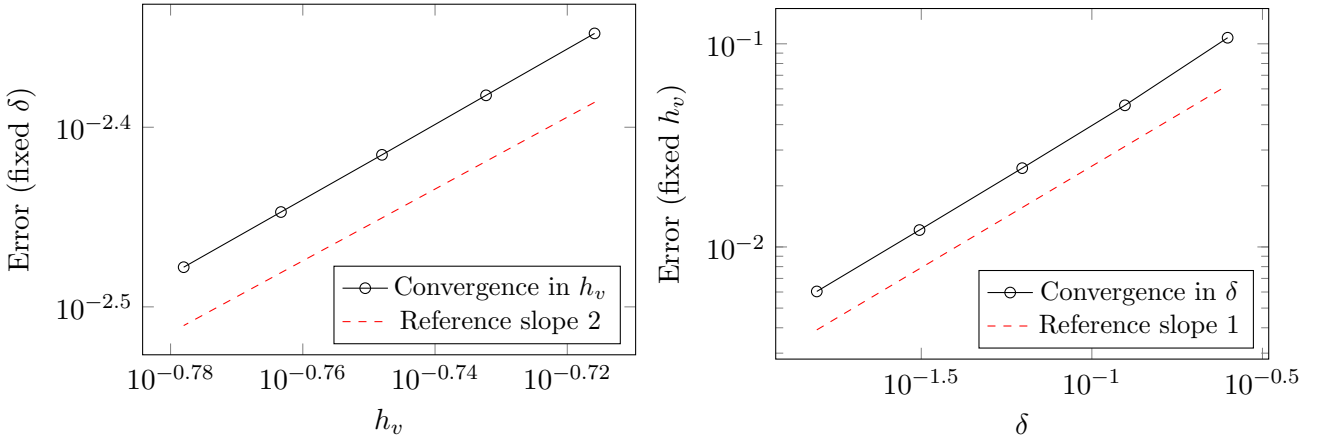
\begin{figure}[ht]
  \centering
  \begin{subfigure}{0.49\textwidth}
    \centering
    \begin{tikzpicture}
      \begin{loglogaxis}[width=.98\linewidth, height=.29\textheight,
        legend pos = south east,legend style={font={\small\arraycolsep=1pt}},
        xlabel = $h_v$, ylabel = Error (fixed $\delta$)]
        \addplot[mark=o,black,solid] table[x index=0, y index=1]  {data/space_convergence.txt};
        \addplot[red, dashed, mark=none] 
        table[x index=0, y index=1] {data/reference_slope_2.txt};
        \legend{Convergence in $h_v$, Reference slope $2$};
      \end{loglogaxis}
    \end{tikzpicture}
    \caption{Convergence in $h_v$ with fixed $\delta = 2^{-7}$.}
    \label{fig1:sub1}
  \end{subfigure}
  \hfill
  \centering
  \begin{subfigure}{0.49\textwidth}
    \centering
    \begin{tikzpicture}
      \begin{loglogaxis}[width=.98\linewidth, height=.29\textheight,
        legend pos = south east,legend style={font={\small\arraycolsep=1pt}},
        xlabel = $\d$, ylabel = Error (fixed $h_v$)]
        \addplot[mark=o,black,solid] table[x index=0, y index=1]  {data/time_convergence.txt};
        \addplot[red, dashed, mark=none] 
        table[x index=0, y index=1] {data/reference_slope_1.txt};
        \legend{Convergence in $\d$, Reference slope $1$};
      \end{loglogaxis}
    \end{tikzpicture}
    \caption{Convergence in $\d$ with fixed $h_v = 0.05$.}
    \label{fig1:sub2}
  \end{subfigure}

  \caption{
  $L_2$-convergence order in $h_v$ and $\delta$ from \eqref{L_2_err_approx}, computed using 100 Monte Carlo samples with $\alpha = 1$ and $\sigma = 0.1$. The plots show convergence with respect to $h_v$ (left) and $\delta$ (right).
  }
  \label{fig1}
\end{figure}

\begin{figure}[t!]
	\centering
	\begin{subfigure}{0.49\textwidth}
		\centering
		\includegraphics[width=\linewidth]{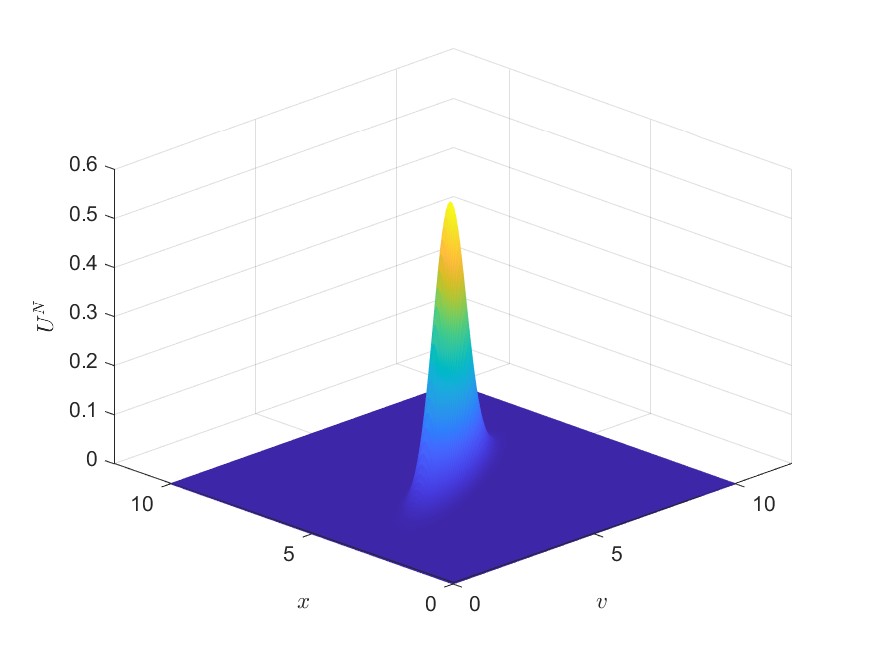}
		\caption{Approximate solution with $h_v = 10^{-1}$, $\delta = 10^{-2}$.}
		\label{fig2:sub1}
	\end{subfigure}
	\begin{subfigure}{0.49\textwidth}
		\centering
		\includegraphics[width=\linewidth]{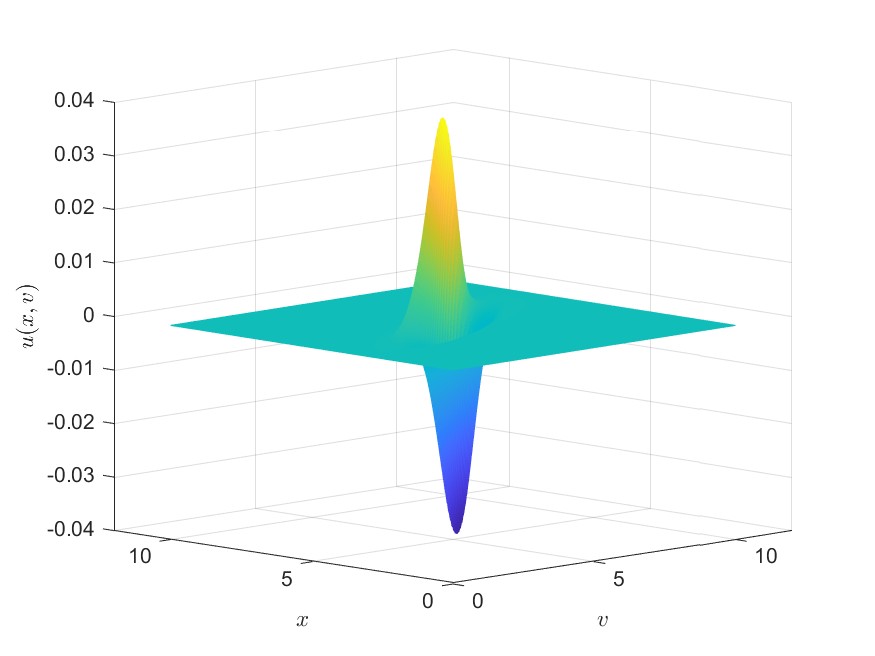}
		\caption{Difference between approximate and exact solution.}
		\label{fig2:sub2}
	\end{subfigure}
	
	\caption{Numerical solution at $T=1$ (left) with $h_v = 10^{-1}$ and $\delta = 10^{-2}$, and pointwise error with respect to the exact solution (right). The initial conditions are $x_0 = 0.1$ and $v_0 = 5$.}
	\label{fig2}
\end{figure}

\subsection{Comparison of MC and MLMC for the numerical solution of the SPDE}

In this experiment, we investigate the efficiency of the Multilevel Monte Carlo (MLMC) method compared to the standard Monte Carlo (MC) approach for approximating the expected value of a functional of the solution of the SPDE. We performed two tests, corresponding to different spatial resolutions: $n = 4$ and $n = 8$, with $\eps = n^{-2}$. Consequently, the number of independent runs for the standard MC approach was set as $D = \eps^{-2}$, resulting in $D = 256$ for $n = 4$ and $D = 4096$ for $n = 8$.

For the MLMC simulations, we first estimated the constants $c_1$ and $c_2$ from Theorem \ref{MLMC_theorem}. Based on preliminary tests:

\begin{itemize}
    \item For $n = 4$, we obtained $L = 2$ levels with $N_l = [52, 2, 1]$ samples per level.
    \item For $n = 8$, we obtained $L = 3$ levels with $N_l = [1637, 52, 2, 1]$ samples per level.
\end{itemize}

We want to point out that once the constants $c_1$ and $c_2$ have been estimated, there is no need to repeat the parameter estimation procedure when using different accuracy levels, as the values of $L$ and $N_l$ can be scaled accordingly.

The discretization parameters for each level $l = 0,\dots,L$ were set as
\begin{equation*}
\tau^{(l)} = 2^{-l}, \qquad h_{v}^{(l)} = 2^{- l / 2},
\end{equation*}
ensuring a proper scaling of the temporal and spatial resolutions across levels.

The quantity of interest, or \emph{payoff function}, was chosen as the $L_2$-norm of the SPDE solution, $P = \|U\|$. The reference exact value $P_\text{exact}$ was computed as the mean over all $D$ runs, yielding
\begin{equation*}
P_\text{exact} = 0.875063 \quad \text{for } n=4, \qquad 
P_\text{exact} = 0.875063 \quad \text{for } n=8.
\end{equation*}

The results, including the mean estimate, absolute and relative errors, and computational time, are summarized in Table~\ref{tab:MC_MLMC_results}. For MLMC, the reported time includes the additional time required to estimate the parameters $L$ and $N_l$.

\begin{table}[t!]
\centering
\caption{Comparison of standard Monte Carlo (MC) and Multilevel Monte Carlo (MLMC) for the SPDE test cases with $n=4$ and $n=8$. The absolute and relative errors are computed with respect to the reference value $P_{\text{exact}}$. For MLMC, the time reported in parentheses corresponds to the additional (one-time) cost of estimating the constants $c_1,c_2$ and the parameters $L$, $N_l$, which does not need to be repeated for different target accuracies $\varepsilon$.}
\label{tab:MC_MLMC_results}
\begin{tabular}{lcccc}
\toprule
Method & $\hat{Y}$ & Absolute Error & Relative Error & Time (s) \\
\midrule
\multicolumn{5}{c}{\textbf{Test 1: $n = 4$, $D = 256$, $L = 2$, $N_l = [52,2,1]$}} \\
Monte Carlo & 0.931161 & 0.0560985 & 0.064108 & 13.02 \\
Multilevel Monte Carlo & 0.907252 & 0.0321893 & 0.0367851 & 0.10 (+10.61) \\
\midrule
\multicolumn{5}{c}{\textbf{Test 2: $n = 8$, $D = 4096$, $L = 3$, $N_l = [1637,52,2,1]$}} \\
Monte Carlo & 0.896013 & 0.0209506 & 0.0239419 & 22509.72 \\
Multilevel Monte Carlo & 0.886517 & 0.0114547 & 0.0130902 & 6.15 (+10.61) \\
\bottomrule
\end{tabular}
\end{table}

\section{Conclusions}

In this work we studied the numerical approximation of degenerate Langevin-type SPDEs in two spatial dimensions, a class of equations that arises naturally in mathematical finance and in the filtering of partially observed diffusions. To handle the mixed deterministic-stochastic structure of the operator $\mathbf{Y} := \partial_t + v\partial_x$ together with the degeneracy of the diffusion in the $x$ direction, we introduced a semi-implicit Milstein finite difference scheme, treating the drift term implicitly and the diffusion term with a Milstein-type explicit correction.

Through a Fourier analysis of the scheme, we derived explicit conditions on the coefficients $a$ and $\sigma$ (namely $\sigma^2 < a$ and $2\sigma^4 < a^2$) under which the scheme is asymptotically mean-square stable, and we established a full convergence analysis, proving that the scheme achieves order one convergence in the time step $\delta$ and order two convergence in the velocity mesh size $h_v$, in the discrete $L^2$-in-space, $L^2$-in-probability norm. These theoretical rates were confirmed numerically, with the observed convergence orders matching the predicted first-order behavior in $\delta$ and second-order behavior in $h_v$.

We then embedded the proposed scheme within a Multilevel Monte Carlo (MLMC) framework in order to reduce the computational cost of estimating expectations of functionals of the SPDE solution. Exploiting the weak and strong error estimates obtained from the Fourier analysis, we showed that the MLMC complexity for achieving a root-mean-square error of order $\varepsilon$ is $\mathcal{O}(\varepsilon^{-3})$, a substantial improvement over the $\mathcal{O}(\varepsilon^{-5})$ complexity of a standard Monte Carlo approach. This theoretical gain was corroborated by numerical experiments, which showed that MLMC achieves comparable or better accuracy than standard Monte Carlo at a fraction of the computational time, with the advantage becoming more pronounced as the target accuracy increases.

Overall, our results demonstrate that combining a semi-implicit Milstein discretization with a multilevel Monte Carlo strategy provides an effective and theoretically grounded approach for the simulation of degenerate Langevin-type SPDEs. Several directions remain open for future research, including the extension of the scheme to the case of non-constant (space- and time-dependent) coefficients $a_t(x,v)$ and $\sigma_t(x,v)$, a more refined treatment of boundary effects on bounded domains, and the application of the proposed methodology to the pricing of path-dependent financial derivatives such as Asian options, where the degenerate structure of the SPDE plays a central role.

\section*{Acknowledgements}
The first author is member of the INdAM Research Group GNCS.\\
This research has been mainly developed during a research stay of S. Portaro at University of A Coru\~{n}a, supported by the Marco Polo Scholarship program of the Alma Mater Studiorum - Universit\`a di Bologna. Moroever, S. Portaro gratefully acknowledges the financial support from the Italian Ministry of University and Research (MUR) under the National Recovery and Resilience Plan (NRRP), Mission 4, Component 1, Investment 4.1 \textit{Extension of the number of research doctorates and innovative doctorates for public administration and cultural heritage}, funded by the European Union - NextGenerationEU. Project Title: \textit{Tecniche di Data Science in Infrastrutture HPC per il Monitoraggio Ambientale delle Valli del Reno, Lavino e Samoggia} - CUP: J33C23002500002 - Alma Mater Studiorum - Universit\`a di Bologna.\\
C. V\'azquez acknowledges the funding from the Galician Government through the grant ED431C 2022/047 (both including FEDER financial support). C. V\'azquez also acknowledges the support of CITIC, as a center accredited for excellence within the Galician University System and a member of the CIGUS Network, that receives subsidies from the Department of Education, Science, Universities, and Vocational Training of the Xunta de Galicia. Additionally, CITIC is co-financed by the EU through the FEDER Galicia 2021-27 operational program (Ref. ED431G 2023/01).

\bibliographystyle{acm}
\bibliography{bib1}

@article{Reisinger_Wang2019,
author = {Christoph Reisinger and Zhenru Wang},
title = {Stability and error analysis of an implicit {M}ilstein finite difference scheme for a two-dimensional {Z}akai {SPDE}},
journal = {BIT Numerical Mathematics},
year = {2019},
volume = {59},
publisher = {Springer Nature},
month = {jun},
url = {https://doi.org/10.1007/s10543-019-00761-8},
number = {4},
pages = {987--1029},
doi = {10.1007/s10543-019-00761-8}
}

@article{Polidoro_Mogavero1995,
author = {Polidoro, Sergio and Mogavero, C.},
year = {1995},
month = {12},
pages = {193-205},
title = {A finite difference method for a boundary value problem related to the {K}olmogorov equation},
volume = {32},
journal = {Calcolo},
doi = {10.1007/BF02575835}
}

@article{Pascucci_Pesce2022,
  title={Backward and forward filtering under the weak {H}{\"o}rmander condition},
  author={Andrea Pascucci and A. Pesce},
  journal={Stochastics and Partial Differential Equations: Analysis and Computations},
  year={2022},
  volume={11},
  pages={177-210},
  url={https://api.semanticscholar.org/CorpusID:257633845}
}

@article{Di_Francesco_Pascucci2004,
 ISSN = {13645021},
 URL = {http://www.jstor.org/stable/4143207},
 author = {Marco di Francesco and Andrea Pascucci},
 journal = {Proceedings: Mathematical, Physical and Engineering Sciences},
 number = {2051},
 pages = {3327--3338},
 publisher = {The Royal Society},
 title = {On the Complete Model with Stochastic Volatility by {H}obson and {R}ogers},
 urldate = {2025-06-11},
 volume = {460},
 year = {2004}
}

@article{Pascucci_Pesce2019,
author = {Pascucci, Andrea and Pesce, Antonello},
year = {2022},
title = {On stochastic {L}angevin and {F}okker-{P}lanck equations: the two-dimensional case},
journal={Journal of Differential Equations},
volume={310},
pages={443-483},
doi = {10.1016/j.jde.2021.11.004}
}

@article{Giles_Reisinger2012,
author = {Giles, Michael B. and Reisinger, Christoph},
title = {{Stochastic finite differences and multilevel Monte Carlo for a class of SPDEs in finance}},
journal = {SIAM Journal on Financial Mathematics},
volume = {3},
number = {1},
pages = {572-592},
year = {2012},
doi = {10.1137/110841916},

URL = { 
    
        https://doi.org/10.1137/110841916
    
    

},
eprint = { 
    
        https://doi.org/10.1137/110841916
    
    

}
}

@article{Kamm_Pagliarani_Pascucci2023,
title = {Numerical solution of kinetic {SPDE}s via stochastic {M}agnus expansion},
journal = {Mathematics and Computers in Simulation},
volume = {207},
pages = {189-208},
year = {2023},
issn = {0378-4754},
doi = {https://doi.org/10.1016/j.matcom.2022.12.029},
url = {https://www.sciencedirect.com/science/article/pii/S037847542200516X},
author = {Kevin Kamm and Stefano Pagliarani and Andrea Pascucci}
}

@book {Pascucci_book2011,
    AUTHOR = {Pascucci, Andrea},
     TITLE = {P{DE} and martingale methods in option pricing},
    SERIES = {Bocconi \& Springer Series},
    VOLUME = {2},
 PUBLISHER = {Springer, Milan; Bocconi University Press, Milan},
      YEAR = {2011},
     PAGES = {xviii+719},
      ISBN = {978-88-470-1780-1},
   MRCLASS = {91-02 (35K10 60G44 60G51 60H30 65-02 91G20 91G60 91G80)},
  MRNUMBER = {2791231},
MRREVIEWER = {Johan Tysk},
       DOI = {10.1007/978-88-470-1781-8},
       URL = {https://doi-org.ezproxy.unibo.it/10.1007/978-88-470-1781-8},
}

@Article{Giles2008a,
journal={Operations Research},
author={Michael B. Giles},
title={Multilevel Monte Carlo Path Simulation},
year={2008},
month={June},
pages={607-617},
volume={56},
number={3},
doi={10.1287/opre.1070.0496},
url={https://ideas.repec.org/a/inm/oropre/v56y2008i3p607-617.html},
}

@book{Lord_Powell_Shardlow_2014,
place={Cambridge},
series={Cambridge Texts in Applied Mathematics},
title={An Introduction to Computational Stochastic PDEs},
publisher={Cambridge University Press},
author={Lord, Gabriel J. and Powell, Catherine E. and Shardlow, Tony},
year={2014},
collection={Cambridge Texts in Applied Mathematics}
}

@book{LiuRoeckner2015,
  author    = {Wei Liu and Michael R{\"o}ckner},
  title     = {Stochastic Partial Differential Equations: An Introduction},
  series    = {Universitext},
  publisher = {Springer Cham},
  year      = {2015},
  doi       = {10.1007/978-3-319-22354-4},
  isbn      = {978-3-319-22354-4}
}

@book{JentzenKloeden2011,
  author    = {Arnulf Jentzen and Peter E. Kloeden},
  title     = {Taylor Approximations for Stochastic Partial Differential Equations},
  publisher = {Society for Industrial and Applied Mathematics},
  year      = {2011},
  doi       = {10.1137/1.9781611971171}
}

\end{document}